\documentclass[a4paper,10pt]{amsart}
\usepackage{amsmath,amsthm,latexsym,amscd,amssymb}
\usepackage[all]{xy}
\usepackage{colortbl}
\usepackage{hyperref}
\numberwithin{equation}{section}
\DeclareMathOperator{\spfl}{sf}

\newtheorem{prop}{Proposition}[section]
\newtheorem{lem}[prop]{Lemma}
\newtheorem{dfn}[prop]{Definition}
\newtheorem{theorem}[prop]{Theorem}

\newtheorem{ass}[prop]{Assumption}
{\par \vspace{0cm}\ \\ \noindent \%\color{blue}\% \begin{tiny}}{\par\end{tiny}\%\color{black}\%\vspace{0cm}}
{\par \vspace{0cm}\ \\ \noindent \%\color{red}\% \begin{tiny}}{\par\end{tiny}\%\color{black}\%\vspace{0cm}}

\def\ds{\displaystyle}

\def\dim{\mathop{\rm dim}}

\def\Im{\mathop{\rm Im}}

\def\vol{\mathop{\rm vol}}

\def\End{\mathop{\rm End}}

\def\Cinf{C^{\infty}}
\def\Ker{\mathop{\rm Ker}}

\def\dom{\mathop{\rm Dom}}

\def\ov{\overline}

\def\id{\mathop{\rm id}}
\def\ra{\partial}
\def\ten{\otimes}

\def\phi{\varphi}
\def\de{\delta}

\def\ve{\varepsilon}

\def\a{\alpha}

\def\O{\Omega}
\def\b{\beta}
\def\la{\lambda}

\def\s{\sigma}
\def\t{\tau}

\def\R{\mathbin{\mathbb R}}
\def\Z{\mathbin{\mathbb Z}}
\def\bbbn{\mathbin{\mathbb N}}

\def\C{\mathbb{C}}

\def\cA{{\mathcal A}}

\def\cD{{\mathcal D}}\def\cE{\mathcal E}\def\cF{{\mathcal F}}\def\cR{{\mathcal R}}
 
\def\cN{\mathcal N}
\def\cJ{\mathcal J}
\def\cM{\mathcal M}
\def\fD{\mathfrak D}

\newcommand{\n}[1]{\left\| #1\right\|}

\def\tr{\mathop{\rm tr}\nolimits}

\def\trl{\mathop{\rm tr_\Lambda}}

\def\Tr{\mathop{\rm Tr}\nolimits}

\def\supp{\mathop{\rm supp}}

\def\ind{\mathop{\rm ind}\nolimits}

\newcommand\Di{D\kern-7pt/}

\title{Spectral flow, index and the signature operator}

\author{Sara Azzali}
\address{Sara Azzali\\Mathematisches Institut\\Bunsenstr. 3-5 \\37073 G\"ottingen\\
Germany}
\email{azzali@uni-math.gwdg.de}
\thanks{}

\author{Charlotte Wahl}
\address{Charlotte Wahl\\Leibniz-Archiv\\Waterloostr. 8\\ 30169 Hannover\\ Germany}
\email{wahlcharlotte@googlemail.com}
\thanks{}

\begin{document}
\begin{abstract}
We relate the spectral flow to the index for paths of selfadjoint Breuer-Fredholm operators affiliated to a semifinite von Neumann algebra, generalizing results of Robbin-Salamon and Pushnitski. Then we prove the vanishing of the von Neumann spectral flow for the tangential signature operator of a foliated manifold when the metric is varied. We conclude that the tangential signature of a foliated manifold with boundary does not depend on the metric. 
In the Appendix we reconsider integral formulas for the spectral flow of paths of bounded operators.
\end{abstract}
\maketitle

\tableofcontents

\section{Introduction}

Since its introduction by Breuer \cite{B1,B2}, index theory in von Neumann algebras has been extensively developed, motivated by the geometric situations of coverings and foliations. The foundations can be traced back in the work of Atiyah on the \mbox{$L^2$-index} theorem \cite{A}, and then in the index theorem of Connes for measured foliations \cite{Co}. 
The corresponding index theorems for manifolds with boundary with Atiyah-Patodi-Singer boundary conditions are a work of Ramachandran \cite{Ram}. Cylindrical ends counterparts have been proven for coverings by Vaillant \cite{Va} and for foliations by Antonini \cite{An}, who applied their results to the definition and study of $L^2$-signatures for coverings of manifolds with boundary and tangential signatures for measured foliations with boundary \cite{An2}, respectively. 

In parallel the notion of \emph{spectral flow} for paths of selfadjoint Breuer-Fredholm operators affiliated to a von Neumann algebra has been the subject of many investigations. The definition of the spectral flow in type II von Neumann algebras is due to Phillips \cite{Ph}. It is based on the idea that the spectral flow of a path of such operators $(D_u)_{u \in [0,1]}$ measures the discontinuity of the path of projections onto the positive part of the spectrum. His definition applies to continuous paths of operators with respect to the Riesz topology. A more general definition was given in terms of a winding number in \cite{Wa}. Both approaches have been used to prove integral formulas for the spectral flow, see for example \cite{CP}\cite{Wa}\cite{CPS}.  

A natural question is whether, as in the classical case, the \emph{spectral flow} of a path of affiliated operators $(D_u)_{u\in [0,1]}$ equals the \emph{von Neumann index} of the operator $\partial_u+D_u$ (with Atiyah-Patodi-Singer conditions when the endpoints are not invertible). 

For Dirac operators on a closed odd-dimensional manifold this follows for example from the variational formula for $\eta$-invariants \cite{Me}. For a fixed Dirac operator which has been perturbed by a path of endomorphisms, a corresponding result in the $L^2$-index theory for coverings was established in \cite{bcp}. In the classical situation, the equality ``index = spectral flow'' was proven for a path of general selfadjoint operators with compact resolvents via an axiomatic approach in \cite{RS}. The conditions on the path were further relaxed in \cite{Ra}. 
A noncommutative ($C^*$-algebraic) version has been derived in \cite{LP} for Dirac operators, and for more general operators on Hilbert $C^*$-modules in \cite{wans}.

The first result of this paper is the equality ``spectral flow = index'' for a path of selfadjoint operators with common domain, and resolvents in the ideal $K(\cN)$, where $\mathcal N$ is a semifinite von Neumann algebra. In particular, Theorem \ref{spflind} deals with the case when the endpoints are invertible. We use only few properties of the spectral flow and the index for the proof (namely homotopy invariance, additivity with respect to concatenation of paths and to direct sums, and a normalization property). It is based on the ideas of the proofs in the noncommutative case in \cite{LP} and \cite{wans}. Our approach can also be used to generalize these results further, see the remarks following Lemma \ref{afffred}. Our proof is different from the one in \cite{RS}\cite{Ra}, which makes use of the discreteness of the spectrum in the classical case and thus does not generalize to our situation.

Along the way we also get the equality $\spfl((D+K_u)_{u \in [0,1]})=\ind(\ra_u + D + K_u) $ for a path of symmetric operators $K_u$ which are relatively bounded with respect to a selfadjoint Breuer-Fredholm operator $D$ with bounds $\alpha_u<1$ and for which $(K_uD^{-1})_{u \in [0,1]}$ is a continuous path in $K(\cN)$ (Prop. \ref{push}). Here we do not assume that the resolvents of $D$ are in $K(\cN)$. This is indeed another main result of the paper, which generalizes a result of Pushnitski \cite{pu}. Note that to this aim Lemmas \ref{D+Kestimate} and \ref{inverse} are proven under more general conditions than necessary just for the proof of Theorem \ref{spflind}.

Theorem \ref{spfl=ind APS} proves the equality for a path with non-invertible endpoints using the Atiyah-Patodi-Singer index. Here we show that the unbounded operator with Atiyah-Patodi-Singer boundary conditions $(\partial_u+D_u)^{APS}$ is Breuer-Fredholm. This is done essentially as in \cite{Ram}.

Next we translate our results to the \emph{type II} geometric situation of foliated manifolds admitting a holonomy invariant transverse measure (Prop. \ref{spflind geometric}).
We investigate in particular the tangential signature operator. In the classical case of the signature operator on a closed manifold, it is well known that a variation of the metric on the manifold does not produce spectral flow. This comes from the strong link between the kernel of the signature operator and the cohomology of the manifold.
Using the integral formula for the spectral flow (\cite{CP,Wa}) we get an analogous result for foliated manifolds (Prop. \ref{sf=0}).
The main step in the proof is based on a beautiful lemma by Cheeger and Gromov (Lemma \ref{lem-cg}), which translates the cohomological nature of the kernel of the signature operator into an analytic property. Our proof also yields a reinterpretation of the result of Cheeger and Gromov in terms of spectral flow.

As an application of our results, we prove that the measured analytic signature of a foliated manifold with boundary, which was defined in \cite{An}, is independent of the metric. 
This follows from the vanishing spectral flow for the signature operator and from a von Neumann relative index theorem by similar arguments as in \cite{LP}. 

Our methods also apply to $L^2$-signatures for manifolds with boundary, which have been defined in \cite{Va} and studied further in \cite{LS}. However, in this case the vanishing of the $L^2$-spectral flow for the signature operator is a rather direct consequence of the vanishing of the ordinary spectral flow, and the independence of the $L^2$-$\rho$-invariants of the metric. 

Appendix \ref{integral formula} is related to but independent of the main body of the paper. Here we contribute to the investigation of integral formulas for the spectral flow. The integral formulas of Carey and Phillips \cite{CP} for bounded perturbations of a fixed selfadjoint operator with resolvents in $K(\cN)$ have been generalized to paths of selfadjoint operator with resolvents in $K(\cN)$ and common domain in \cite{Wa} by exploiting the relation between the spectral flow and the winding number. By using a different approach the latter result was further extended in \cite{CPS}, in particular such that it applies to paths of bounded operators as well. In the Appendix we show that the proof of \cite{Wa} can be modified to include paths of bounded operators. Our result is not equivalent to the one in \cite{CPS}, nor is one a generalization of the other: indeed, we only assume that the path is strongly differentiable in a certain sense, while in \cite{CPS} the path has to be differentiable in the norm topology. On the other hand, the conditions on the interplay between summability and dependence on the parameter that are imposed in \cite{CPS} are weaker than ours.

\smallskip
{\it Acknowledgements:}
We would like to thank Moulay Benameur for interesting discussions. The second named author thanks him for an invitation to Metz, which started our collaboration. We also thank Paolo Piazza for drawing our attention to the question about the relation between spectral flow and index.

\section{Spectral flow and index}\label{sf=ind}
In this section we prove the equality between \emph{spectral flow} and \emph{index} (Theorems \ref{spflind} and \ref{spfl=ind APS}) in the general context of semifinite von Neumann algebras. 

Let $\cN$ be a von Neumann algebra acting on a separable Hilbert space $H$ and endowed with a faithful normal semifinite trace. The ideal $K(\cN)$ in $\cN$ is the smallest closed ideal containing all elements of finite trace.

There is an induced semifinite trace on the von Neumann algebraic tensor product $\cM:=B(L^2(\R)) \ten \cN$, which acts on $L^2(\R) \ten H\cong L^2(\R,H)$. In the following the index $\ind$ is defined with respect to the semifinite von Neumann algebra $\cM$, whereas the spectral flow $\spfl$ is defined with respect to $\cN$.

For the theory of Breuer-Fredholm operators in this setting we refer to \cite{CPRS}. We will in particular use the criterion that an unbounded operator affiliated to a semifinite von Neumann algebra is Breuer-Fredholm if it has a right and a left parametrix. This assertion is proven in Lemma 3.15 in \cite{CPRS} under the condition that the right and left parametrix agree. However the proof works also if they do not agree.

In general, for a closed operator $A$ on a Hilbert space $V$ we denote by $H(A)$ the space $\dom A$ endowed with the scalar product $\langle x,y\rangle_V + \langle Ax,Ay\rangle_V$.
This is a Hilbert space.

\begin{theorem}
\label{spflind}
Let $(D_u)_{u \in [0,1]}$ be a path of selfadjoint operators affiliated to $\cN$ with common domain and resolvents in $K(\cN)$. We assume that $D_u$ depends continuously on $u$ as a bounded operator from $H(D_0)$ to $H$ (with respect to the operator norm). Furthermore we assume that the endpoints $D_0,D_1$ are invertible. Then
$$\spfl((D_t)_{u \in [0,1]})=\ind (\ra_u + D_u) \ .$$
\end{theorem}

We divide the proof of the theorem into several lemmata. First we fix some conventions, which will be tacitly applied in similar situations throughout the paper.

We extend $(D_u)_{u \in [0,1]}$ to a path on $\R$ by setting $D_u=D_0$ for $u<0$ and $D_u=D_1$ for $u>1$. We set $\cD:=\ra_u+D_u$, which we understand as a closed operator on $L^2(\R,H)$ having $\Cinf_c(\R,H)$ as a core. 

For an open interval $U=(u_0,u_1) \subset \R$ we define the path 
$$D^U_u=\left\{\begin{array}{ll} D_{u_0} & u \le u_0 \\
D_u & u \in (u_0,u_1)\\
D_{u_1} & u \ge u_1 \ .
\end{array}\right.$$
We set $\cD^U=\ra_u + D^U$.

Define the symmetric closed operator $$\tilde \cD=\left(\begin{array}{cc} 0 & -\ra_u + D_u \\ \ra_u + D_u & 0 \end{array}\right) \ .$$ 
Furthermore 
for fixed $x \in \R$ set $\cD_x=\ra_u + D_x$,
$$\tilde \cD_x=\left(\begin{array}{cc} 0 & -\ra_u + D_x \\ \ra_u + D_x & 0 \end{array}\right)$$
and $$\tilde D_x=\left(\begin{array}{cc} 0 &  D_x \\ D_x & 0 \end{array}\right) \ .$$

The following technical lemma will be very useful.

\begin{lem}
\label{compinv}
Let $D$ be a selfadjoint invertible operator on $H$ affiliated to $\cN$ and let $K \in \cN$ be such that $KD^{-1} \in K(\cN)$. Let $\phi \in \Cinf_c(\R)$. Then 
$$\phi K (\pm \ra_u + D)^{-1} \in K(\cM) \ .$$
\end{lem}

\begin{proof}
By Fourier transform $\pm \ra_u + D$ is unitarily equivalent (in $\cM$) to $\pm it + D$, which has an inverse in $\cM$. It follows that $\pm \ra_u + D$ has an inverse in $\cM$, and thus is affiliated to $\cM$. 

Any element in $L^2(\R^2)$ defines a compact integral operator on $L^2(\R)$, and by tensoring with the identity an element in $\cM$.

The key observation is that in a similar way any element in $L^2(\R^2,K(\cN))$ acts on $L^2(\R,H)$ defining an element in $K(\cM)$.

Via Fourier transform the operator $\phi K (\pm \ra_u + D)^{-1}$ on $L^2(\R,H)$ is unitarily equivalent to the operator
$$v \mapsto \int_{\R}\hat \phi(x-y) K (\pm iy+D)^{-1}v(y) ~dy \ .$$ 

The assertion follows now since $\hat \phi(x-y) K (\pm iy+D)^{-1} \in L^2(\R^2,K(\cN))$.
\end{proof}

The following lemma shows that the statement of the theorem is well-defined.

\begin{lem}
\label{afffred}
It holds that $\dom \cD=\dom (\ra_u+ D_0)$.

The operator $\cD$ is affiliated to $\cM$ and is Breuer-Fredholm.
\end{lem}

\begin{proof}
Using Fourier transform as in the proof of the previous lemma one sees that $\tilde \cD_x \pm i$ is invertible. The inverse is in $M_2(\cM)$. Thus $\tilde \cD_x$ is selfadjoint and affiliated to $M_2(\cM)$.

By Fourier transform, one also checks that $\tilde D_x (\tilde\cD_x+i)^{-1}$ is bounded. Since 
$$\tilde D_y(\tilde\cD_x+i)^{-1}=(\tilde D_y\tilde D_x^{-1})\tilde D_x(\tilde\cD_x+i)^{-1}$$ 
is bounded, it follows that $(\tilde\cD_y-\tilde\cD_x)(\tilde\cD_x+i)^{-1}$ is bounded, and hence $(\tilde\cD_y+i)(\tilde\cD_x+i)^{-1}$ is bounded for all $x,y \in [0,1]$. This implies that $\dom \tilde\cD_x=\dom \tilde\cD_0$. Note also that $\tilde D_y(\tilde\cD_x+i)^{-1}$ depends continuously on $y$. 

Similarly it follows that $\tilde \cD(\tilde \cD_0+i)^{-1}$ is bounded. Thus $\dom \tilde \cD_0 \subset \dom \tilde \cD$. 

In the following we show that $\tilde \cD+i:H(\tilde \cD_0) \to L^2(\R,H\oplus H)$ has a bounded inverse, which is in $M_2(\cM)$. Then it follows that $\tilde \cD$ is selfadjoint affiliated to $\cM$ with $\dom \tilde \cD = \dom \tilde \cD_0$.

First we make the following assumption on the path $(D_u)_{u \in [0,1]}$: 

\begin{ass}
\label{ass} 
Let $C$ be the norm of $(\tilde \cD_0 - i)^{-1}$ as an operator from $L^2(\R,H\oplus H)$ to $H(\cD_0)$. 
We assume that the operator
$\tilde \cD-\tilde \cD_0:H(\tilde \cD_0) \to L^2(\R,H\oplus H)$ is bounded by $\frac{1}{4C}$.
\end{ass}

Under Assumption \ref{ass} the Neumann series
$$(\tilde \cD - i)^{-1}=(\tilde \cD_0 - i)^{-1}\sum_{n=0}^{\infty}\bigr((\tilde \cD_0- \tilde \cD)(\tilde \cD_0 - i)^{-1}\bigl)^n$$
converges and defines an element in $M_2(\cM)$. Furthermore $\cD_0(\tilde \cD - i)^{-1}$ is well-defined and bounded. Thus $\dom \tilde \cD\subset \dom \tilde \cD_0$. It follows that $\tilde \cD$ is selfadjoint and affiliated to $M_2(\cM)$.

Now we drop Assumption \ref{ass}. For an open interval $U$ we define $\tilde \cD^U, \cD^U_x, \tilde \cD^{U}_x$ as above.

Let $(U_i)_{i=0, \dots, k+1},~k \in \bbbn,$ be a finite covering of $\R$ by open intervals with $(-\infty,0)= U_0$, $(1,\infty) = U_{k+1}$ and with $U_i, ~i=1, \dots, k$ precompact. We assume that the sets $U_i, i=1, \dots, k$ are small enough such that for each $i$ the path $D^{U_i}_u$ fulfills Assumption \ref{ass}. By the compactness of $[0,1]$ such a covering exists. By the previous argument, the operator $\tilde \cD^{U_i}$ with $\dom \tilde \cD^{U_i}=\dom \tilde \cD_0$ is selfadjoint and affiliated to $M_2(\cM)$.

Let $(\chi_i^2)_{i=0,\dots, k+1}$ be a partition of unity subordinate to the covering with $\chi_i \in \Cinf(\R)$.

For $\lambda \in i\R \setminus \{0\}$
we set
$$Q(\lambda)=\sum_{i=0}^{k+1}\chi_i(\tilde \cD^{U_i} - \lambda)^{-1}\chi_i \ .$$

It holds that
$$K :=(\tilde \cD -\lambda)Q(\lambda)-1=\sum_{i=0}^{k+1}\chi_i'(\tilde \cD^{U_i}-\lambda)^{-1}\chi_i  \ .$$
By choosing $|\lambda|$ large enough we ensure that $1+K$ is invertible and thus $Q(\lambda)(1+K)^{-1}$ is a right inverse of $\tilde \cD-\lambda$. Note that $K, Q(\lambda)\in M_2(\cM)$. Similarly one constructs a left inverse. It follows that $\tilde \cD$ with $\dom \tilde \cD=\dom \tilde \cD_0$ is selfadjoint and affiliated to $M_2(\cM)$. 

Now we show that $\tilde \cD$ is Breuer-Fredholm. 

Define
$$Q= \chi_0\tilde \cD_0^{-1}\chi_0 + \chi_{k+1} \tilde \cD_1^{-1}\chi_{k+1}+ \sum_{i=1}^{k}\chi_i(\tilde \cD^{U_i} - i)^{-1}\chi_i \ .$$
It holds that
\begin{align*}
\tilde \cD Q-1&=\chi_0'\tilde \cD_0^{-1}\chi_0 + \chi_{k+1}'\tilde \cD_1^{-1}\chi_{k+1} + i\sum_{i=1}^{k}\chi_i(\tilde \cD^{U_i}-i)^{-1}\chi_i\\
&\quad +\sum_{i=1}^{k}\chi_i'(\tilde \cD^{U_i}-i)^{-1}\chi_i  \ .
\end{align*}
Using Lemma \ref{compinv} one checks that the right hand side is in $M_2(K(\cM))$. Thus $Q$ is a right parametrix of $\tilde \cD$. A similar calculation yields that $Q$ is also a left parametrix of $\tilde \cD$.
\end{proof}

The method of the previous lemma works also in a $C^*$-algebraic context and allows to generalize Theorem 10 in \cite{LP} as well as Lemma 3.13 and Prop. 3.15 in \cite{wans}: the path $(D_u)_{u \in [0,1]}$ considered in \cite{wans} was a path of regular selfadjoint operators with common domain and compact resolvents on the standard Hilbert $\cA$-module $H_{\cA}$, where $\cA$ is a unital $C^*$-algebra. It was assumed that $D_u-D_0$ is bounded. This may now be replaced with the condition that $D_u:H(D_0) \to H_{\cA}$ depends continuously on $u$. See the remarks following Lemma 3.13 and Prop. 3.15 in \cite{wans}. A similar statement holds for the odd case, see \cite[\S 8]{wans}.

Next we prove the additivity property for the index.
\begin{lem}
\label{cut}
Let $y \in (0,1)$ be such that $D_y$ be invertible. We define $U_0=(-\infty,y), U_1=(y,\infty)$.  Then
$$\ind(\ra_u+D_u)=\ind(\ra_u+ D^{U_0}_u) + \ind(\ra_u + D^{U_1}_u) \ .$$
\end{lem}

\begin{proof}
The proof is an adaption of the proof of Lemma 3.13(5) in \cite{wans} to the present context. That proof in turn is a variation of the proof of the $K$-theoretic relative index theorem (Theorem 1.14 in \cite{bu}).

Define $$\fD=\cD^{U_0} \oplus \cD^{U_1} \oplus \cD^* \oplus \cD_y^* \ .$$ The operator $\fD$ is a Breuer Fredholm operator affiliated to the semifinite von Neumann algebra $M_4(\cM)$. We have to show that its index vanishes. 

By the homotopy invariance of the index we can assume that the path $D_u$ (and thus also $D^{U_0}_u$, $D^{U_1}_u$) is constant on $(y-\delta, y + \delta)$ for some $\delta >0$.

Let $\chi_1:\R \to [0,1]$ be a smooth function such that $\chi_1(x)=1$ for $x< y-\delta/2$ and $\chi_1(x)=0$ for $x>y+\delta/2$ and set $\chi_2=\sqrt{1-\chi_1^2}$. Define
$$X=\left(\begin{array}{cccc} 0 & 0 & -\chi_1 & -\chi_2 \\
 0 & 0 & -\chi_2 & \chi_1 \\
 \chi_1 & \chi_2 & 0 & 0 \\
 \chi_2 & -\chi_1 & 0 & 0 
 \end{array}\right) \in M_4(\cM) \ .$$   
Then $XX^*=X^*X=1$ and $X^*=-X$. Furthermore $X\fD - \fD^* X$ equals
$$\left(\begin{array}{cccc} 0 & 0 & -\chi_1 \cD^* + (\cD^{U_0})^* \chi_1 & -\chi_2 \cD_y^* + (\cD^{U_0})^* \chi_2 \\
0 & 0 & -\chi_2 \cD^* +(\cD^{U_1})^*\chi_2 & \chi_1\cD_0^* - (\cD^{U_1})^* \chi_1 \\  
\chi_1\cD^{U_0} - \cD\chi_1 & \chi_2 \cD^{U_1} - \cD \chi_2 & 0 & 0\\
\chi_2 \cD^{U_0} - \cD_y\chi_2 & -\chi_1 \cD^{U_1} +\chi_1 \cD_y & 0 & 0
\end{array}\right) \ .$$
One checks easily that this is a bounded operator, for example 
$$-\chi_1 \cD^*+(\cD^{U_0})^* \chi_1=[\chi_1,\ra_u] -\chi_1 D_u + D^{U_0}_u\chi_1=-\chi_1' \ .$$
Similarly one gets that $\fD X- X\fD^*$ is bounded. Define the operators $\tilde \fD:=\left(\begin{array}{cc} 0 & \fD^* \\ \fD & 0 \end{array}\right)$ and $ \ov X:=\left(\begin{array}{cc} 0 & X \\ X & 0 \end{array}\right)$. It follows that $[\ov X,\tilde \fD] \in M_8(\cM)$. Furthermore $[\ov X,\tilde \fD](\pm \ra_u+ D_0)^{-1} \in K(M_8(\cM))$.

Let $\chi \in \Cinf(\R)$ be odd, non-decreasing, with $\chi'(0)>0$ and $\chi^2-1 \in \Cinf_c(\R)$ and such that $\chi(\tilde \fD)^2-1 \in K(M_8(\cM))$. As in the proof of Definition 2.4 of \cite{wans} it follows that $$[\ov X,\chi(\tilde \fD)] \in K(M_8(\cM)) \ .$$ The operator $\chi(\fD)$ is implicitly defined by the equality $$\chi(\tilde \fD)=\left(\begin{array}{cc} 0 & \chi(\fD)^* \\ \chi(\fD) & 0 \end{array}\right) \ .$$
Then $X\chi(\fD)-\chi(\fD)^*X \in K(M_4(\cM))$ and $\chi(\fD)X-X\chi(\fD)^* \in K(M_4(\cM))$. 

Now we proceed as in the proof of Lemma 1.15 in \cite{bu}. For $x \in [0,\pi/2]$ we set
$$F_x:=\cos(x)\chi(\fD) + \sin(x) X \ .$$ 
It holds that
\begin{align*}
\lefteqn{F_xF_x^*-1}\\
&=\cos(x)^2\chi(\fD)\chi(\fD)^* +\sin(x)^2 + \cos(x)\sin(x)(\chi(\fD)X - X\chi(\fD)^*)-1 \\
&=\cos(x)^2(\chi(\fD)\chi(\fD)^* -1) + \cos(x)\sin(x)(\chi(\fD)X - X\chi(\fD)^*)\\ 
&\in K(M_4(\cM)) \ ,
\end{align*} 
and similarly $F_x^*F_x - 1 \in K(M_4(\cM))$. Thus $F_x$ is Breuer-Fredholm for any $x$. From the homotopy invariance of the index it follows that $$\ind(\chi(\fD))=\ind(F_0)=\ind(F_{\pi/2})=\ind(X)=0 \ .$$
\end{proof}

Let $D$ be a selfadjoint operator on $H$. Recall that a symmetric operator $K$ on $H$ with $\dom D\subset \dom K$ is called relatively bounded with respect to $D$ with (not unique) bound $\alpha_K$ if there is $c>0$ such that for all $f \in H$
$$\|Kf\| \le \alpha_K\|Df\| + c\|f\| \ .$$

We refer to \cite[\S 1.4]{da} for the theory of relatively bounded perturbations.

\begin{lem}\label{D+Kestimate}
Let $D$ be a selfadjoint operator on $H$. Let $K$ be a symmetric operator on $H$ with $\dom D \subset \dom K$ which is relatively bounded with respect to $D$ with bound $\alpha_K<1$. 

Then $D+K$ is selfadjoint.

Furthermore for any $0<\delta<1$ and $c>0$ there are $\lambda_0, C>0$ such that 
$$\|D(D + K + \lambda i)^{-1}\| \le C$$ 
for all $\lambda> \lambda_0$ and all $K$ as before which fulfill 
$$\|Kf\| \le \delta\|Df\| + c\|f\|, f\in H \ .$$ 

If $D$ is affiliated to $\cN$ and $K(D+i)^{-1} \in \cN$, then $D+K$ is affiliated to $\cN$.  
\end{lem}

\begin{proof}
For any $0<\delta<1$ and $c>0$ there is $\lambda_0>0$ such that for $\delta <\gamma <1$ it holds that
$$\|K(D + \lambda i)^{-1}\| <\gamma$$ for all $\lambda>\lambda_0$ and for all symmetric $K$ with $\dom D \subset \dom K$ which fulfill $\|Kf\| \le \delta \|Df\| + c\|f\|, f\in H$. This follows from the proof of Lemma 1.4.1 in \cite{da}. 

Thus for $\lambda>\lambda_0$ the resolvents $(D+ K + \lambda i)^{-1}$ are in $\cN$ by the Neumann series
$$(D+ K + \lambda i)^{-1}=(D+ \lambda i)^{-1}\sum_{n=0}^{\infty}(-K(D + \lambda i)^{-1})^n \ .$$
This implies that $D+K$ is selfadjoint and affiliated to $\cN$ if $D$ is affiliated to $\cN$ and $K(D+i)^{-1} \in \cN$. Furthermore we get
$$\|D(D+K + \lambda i)^{-1}\| \le C\sum_{n=0}^{\infty}\| K(D + \lambda i)^{-1}\|^n<C\sum_{n=0}^{\infty} \gamma^n \ .$$
\end{proof}

\begin{lem}
\label{inverse}
Let $D$ be a selfadjoint operator on $H$. Let $K$ be a symmetric operator on $H$ with $\dom D \subset \dom K$ which is relatively bounded with respect to $D$ with bound $\alpha_K<1$. Let $D+K$ be invertible.
For $R>0$ we set $P_R:=1_{[-R,R]}(D)$. Then for $R$ large the operator 
$$D+(1-u)K + uP_RKP_R$$ 
is invertible for any $u \in [0,1]$.
\end{lem}

\begin{proof}  
We set $E_u:=D+(1-u)K + uP_RKP_R$. Since $P_RKP_R$ is bounded, for each $u$ and $R$ the operator $(1-u)K + uP_RKP_R$ is relatively bounded with respect to $D$. Let $\alpha_K<1$ be a bound of $K$. Then there is $c>0$ such that for all $u \in [0,1], R>0$ and $f \in H$
\begin{align*}
\|((1-u)K + uP_RKP_R)f\| &\le (1-u)\|Kf\|+ u\|P_RKP_Rf\|\\
&\le (1-u)\|Kf\| + u\|K(P_Rf)\|\\
& \le (1-u)\alpha_K\|Df\| + c(1-u)\|f\| +u\alpha_K\|P_RDf\|  + cu\|P_Rf\| \\
&\le\alpha_K\|Df\|+c\|f\| \ .
\end{align*}
Thus by the previous lemma $D(E_u+\lambda i)^{-1}$ is uniformly bounded in $u$ and $R$ for $\lambda$ large.

It holds that 
\begin{align*}
\lefteqn{(E_0+\lambda i)^{-1} - (E_u+ \lambda i)^{-1}}\\
&=(E_0+ \lambda i)^{-1}(-uK+uP_RKP_R)(E_u+ \lambda i)^{-1}\\
&=-u(E_0+ \lambda i)^{-1}((1-P_R)K(1-P_R)+(1-P_R)KP_R+P_RK(1-P_R))(E_u+ \lambda i)^{-1} \ .
\end{align*}
We show that this term converges to zero for $R \to \infty$ uniformly in $u$.

Since $(D+\lambda i)(E_u+\lambda i)^{-1}$ is uniformly bounded in $u$ and $R$ and $(E_0+\lambda i)^{-1}(D+ \lambda i)$ is bounded, there is $C \in \R$ such that
\begin{align*}
\|(E_0+\lambda i)^{-1} - (E_u+ \lambda i)^{-1}\|
&\le C\bigl(\|(D+\lambda i)^{-1}(1-P_R)K(1-P_R)(D+\lambda i)^{-1}\| \\
& \quad +\|(D+\lambda i)^{-1}(1-P_R)KP_R(D+\lambda i)^{-1}\| \\
& \quad +\|(D+\lambda i)^{-1}P_RK(1-P_R)(D+\lambda i)^{-1}\| \bigr) \ .
\end{align*} 
Consider, for example, 
the term $\|(D+\lambda i)^{-1}P_RK(1-P_R)(D+\lambda i)^{-1}\|$. The operator $(D+\lambda i)^{-1}P_RK$ is the adjoint of $K(D-\lambda i)^{-1}P_R$, which is uniformly bounded in $R$. Clearly $(1-P_R)(D+\lambda i)^{-1}$ converges to zero for $R \to \infty$. The other terms can be treated similarly. Thus $(E_0+\lambda i)^{-1} - (E_u+ \lambda i)^{-1}$ converges to zero for $R \to \infty$ uniformly in $u$.

Since $E_0$ is invertible and the set of invertible selfadjoint operators is open in the gap topology by Prop. 1.7 in \cite{blp}, the operator $E_u$ is invertible for all $u$ if $R$ is large enough.
\end{proof}

\begin{lem}
Assume that $\cN$ is a von Neumann algebra endowed with a finite trace and let $B_0 , B_1 \in \cN$ be two involutions. Set $B_u:=(1-u)B_0 + uB_1$. Then
$$\spfl((B_u)_{u\in [0,1]})=\ind(\ra_u + B_u) \ .$$
\end{lem}

\begin{proof} 
The assertion has been proven already in \cite[\S 5]{bcp} by an explicit calculation. We give a different proof here. We use the properties of the spectral flow and the index to reduce to paths of a particularly simple form. The value of the spectral flow and the index for these elementary paths can be seen as the normalization property of the spectral flow and the index.  
 
First note that $K(\cN)=\cN$, and thus any element in $\cN$ is Breuer-Fredholm and has resolvents in $K(\cN)$. 

By homotopy invariance the spectral flow of a path in $\cN$ and the index only depend on the endpoints.
  
For $i=0,1$ set $P_i^-=\frac 12 (1-B_i)$.

Define the path $(\beta_u)_{u \in [0,1]}$ by $\beta_u=B_0+4u P^-_0$ for $u \in [0,\frac 12]$ and $\beta_u=B_1+4(1-u)P_1^-$ for $u \in [\frac 12, 1]$. Note that $\beta_{1/2}=1$. 

We set $U_0=(0,\frac 12)$ and $U_1=(\frac 12,1)$ and get paths $\beta_u^{U_i}$.

Now by the additivity of the spectral flow with respect to concatenation of paths
$$\spfl((\beta_u)_{u\in [0,1]})=\spfl((\beta^{U_0}_u)_{u\in [0,1]}) + \spfl((\beta^{U_1}_u)_{u\in [0,1]}) \ , $$
and by Lemma \ref{cut}
$$\ind(\ra_u + \beta_u)=\ind(\ra_u + \beta^{U_0}_u) + \ind(\ra_u + \beta^{U_1}_u) \ .$$
Note that $(1-P_i^-)\beta^{U_i}_u(1-P_i^-)=(1-P_i^-)$. Thus, this path does neither contribute to spectral flow nor to the index. Furthermore $P_0^-\beta^{U_0}_uP_0^-=(-1+4u)P_0^-$ for $u \in [0,\frac 12]$ and $P_1^-\beta^{U_1}_uP_1^-=(-1+4(1-u))P_1^-$ for $u \in [\frac 12,1]$. The spectral flow for these paths can be obtained directly from Phillips' definition \cite{Ph}, whereas the index can be easily calculated by determining explicitely kernel and cokernel. We get 
$$\spfl((\beta_u^{U_0})_{u\in [0,1]})=\ind(\ra_u + \beta_u^{U_0})=\tr(P_0^-)$$
and 
$$\spfl((\beta_u^{U_1})_{u\in [0,1]})=\ind(\ra_u + \beta_u^{U_1})=-\tr(P_1^-) \ .$$
\end{proof}

The following proposition is one of the main results of this paper. It generalizes a result from \cite{pu}. 

\begin{prop}\label{push}
Let $D$ be a selfadjoint invertible operator affiliated to $\cN$ and let $(K_u)_{u \in [0,1]}$ be a path of symmetric operators on $H$ with $\dom D \subset \dom K_u$, such that each $K_u$ is relatively bounded with respect to $D$ with bound $\alpha_u<1$ and such that $(K_uD^{-1})_{u \in [0,1]}$ is a continuous path in $K(\cN)$. Furthermore we assume that $K_0=0$ and that $D+K_1$ is invertible. 

Then 
$$\spfl((D+K_u)_{u \in [0,1]})=\ind(\ra_u + D + K_u) \ .$$
\end{prop}

In the classical case $\cN=B(H)$ the conditions of the Proposition simplify: the condition $K(D+i)^{-1} \in K(H)$ means that $K$ is relatively compact with respect to $D$, and this implies that any $\alpha_K>0$ is a bound.

\begin{proof} 
By Prop. 2.2 of \cite{Le} the bounded transform of the path $D+K_u$ depends continuously on $u$. Hence its spectral flow is well-defined.

As in the proof of Lemma \ref{afffred} one gets that $\pm \ra_u + D + K_u$ with domain $\dom(\ra_u + D)$ are affiliated to $\cM$ and adjoint to each other. 

We show that $\ra_u+D+K_u$ is Breuer-Fredholm. For that end, let $\chi_1 \in \Cinf(\R)$ be a positive function such that $\chi_1(x)=1$ for $x<-R$ and $\chi_1(x)=0$ for $x>R$ for some $R > 0$. We set $\chi_2=\sqrt{1-\chi_1^2}$. 

We define $Q=\chi_1 (\ra_u + D)^{-1} \chi_1+ \chi_2(\ra_u + D + K_1)^{-1}\chi_2$.

Then
\begin{align*}
(\ra_u + D + K_u) Q -1&=\chi_1 K_u(\ra_u + D)^{-1}\chi_1 + \chi_2(K_u-K_1)(\ra_u + D+K_1)^{-1}\chi_2\\
&\quad + \chi_1'(\ra_u + D)^{-1}\chi_1 + \chi_2'(\ra_u + D + K_1)^{-1}\chi_2\ .
\end{align*}

Using Lemma \ref{compinv} and an approximation argument one concludes that the first line of the right hand side is in $K(\cM)$. By choosing $\chi_1$ appropriate we can arrange that the norm of 
$$T:=\chi_1'(\ra_u + D)^{-1}\chi_1 + \chi_2'(\ra_u + D + K_1)^{-1}\chi_2$$ is smaller that $\frac 14$. Then $Q(1+T)^{-1}$ is a right parametrix of $\ra_u+D+K_u$. In a similar way one gets a left parametrix. This implies that $\ra_u + D+ K_u$ is Breuer-Fredholm. 

Next we show that it is enough to establish the equality in the proposition for the case where $D$ is bounded. 

Let $P_R=1_{[-R,R]}(D)$.

For $R >0$ large enough the selfadjoint operator $D+ sP_RK_1  P_R + (1-s)K_1$ is invertible for all $s \in [0,1]$ by Lemma \ref{inverse}. We write $\cE_s:=\ra_u + D + sP_R K_u P_R + (1-s)K_u$ and 
$$\tilde \cE_s=\left(\begin{array}{cc} 0 &\cE_s^* \\ \cE_s & 0 \end{array}\right) \ .$$ 
As usual, we write $\cD_0=\ra_t + D$, and define
$\tilde \cD_0$ accordingly. 

As in the proof of Lemma \ref{afffred} one shows that $\dom \tilde \cD_0 = \dom \tilde \cE_s$. Since $\tilde \cE_s:H(\tilde \cD_0) \to L^2(\R,H \oplus H)$ is continuous in $s$, Prop. 2.2 of \cite{Le} implies that the bounded transform of $\tilde \cE_s$ depends continuously on $s$. 
Thus the homotopy invariance of the index implies that 
$$\ind(\ra_u + D + K_u)=\ind(\cE_0)=\ind(\cE_1)=\ind(\ra_u + D + P_R K_u P_R) \ .$$
The operator on the right hand side commutes with $P_R$ and thus is diagonal with respect to the decomposition $L^2(\R,H)=L^2(\R,P_RH) \oplus L^2(\R,(1-P_R)H)$. The index of $\ra_u + (1-P_R)D$, taken with respect to the von Neumann algebra $(1-P_R)\cM(1-P_R)$, vanishes. 

Similarly, we have
$$\spfl((D+K_u)_{u \in [0,1]})=\spfl((D + P_R K_u P_R)_{u \in [0,1]}) \ .$$
Clearly, the spectral flow of the constant path $(1-P_R)D$ vanishes. Thus we only need to prove the assertion for the path $P_R(D + K_u)P_R$ of bounded operators.

Therefore we assume from now on that $D$ is bounded. Then $K_u$ is a continuous path in $K(\cN)$. Define the involution $B_i=1_{[0,\infty)}(D+K_i)-1_{(-\infty,0]}(D+K_i),~i=0,1$. The path $(D+K_u)_{u \in [0,1]}$ is homotopic to the path $(B_u:=B_0 + u(B_1-B_0))_{u \in [0,1]}$ through paths of selfadjoint Breuer-Fredholm operators with invertible endpoints. Furthermore $B_1-B_0 \in K(\cN)$. 

In the following we use ideas from \cite[\S 5]{bcp} in order to reduce to a finite situation, in which we can apply the previous lemma: the operator $B_u$ is invertible for $u \neq \frac 12$. Let $4>\ve>0$. The projection $P_{\ve}:= 1_{[0,\ve]}((B_0+B_1)^2)$ commutes with $B_0,B_1$. Thus $(1-P_{\ve})B_u(1-P_{\ve})$ is invertible for any $u$. It follows that the path $(1-P_{\ve})B_u(1-P_{\ve})$ is homotopic to the constant path $(1-P_{\ve})B_0(1-P_{\ve})$ through paths with invertible endpoints. Thus it neither contributes to the index nor to the spectral flow. Since $$(B_0+B_1)^2-4=B_0B_1+B_1B_0-2=(B_0-B_1)B_1+B_1(B_0-B_1) \in K(\cN) \ ,$$ 
the projection $P_{\ve}$ is finite. For the path $(P_{\ve}B_uP_{\ve})_{u \in [0,1]}$ the assertion follows from the previous lemma. 
\end{proof}

\begin{proof}[Proof of the Theorem] Let $0=u_0 < u_1 < \dots < u_k < u_{k+1}=1$ be such that for $i=0,\dots, k$ there is a selfadjoint operator $K_i \in K(\cN)$ with $D_u+K_i$ invertible for $u \in [u_i,u_{i+1}]$. We also assume that $K_0=0$ and set $K_{k+1}=0$. 

Such a subdivision exists: Let $s \in [0,1]$. For $K=2 1_{[-1,1]}(D_s)$ the operator $D_s + K:H(D_0) \to H$ is invertible. Since $D_u + K:H(D_0) \to H$ depends continuously on $u$, it is invertible in a small neighbourhood of $s$.
Now the existence of the subdivision follows from the compactness of $[0,1]$.
 
We define a path $(E_u)_{u \in [0,1]}$ as follows: We set
$$E_u=\left\{\begin{array}{ll}
D_{2u-u_i}+K_i, &u \in [u_i,\frac{u_i+u_{i+1}}{2}] \ ,\\ D_{u_{i+1}}+\frac{2}{u_{i+1}-u_i}\bigl((u_{i+1}-u)K_i+(u-\frac{u_i+u_{i+1}}{2})K_{i+1}\bigr) , & u \in [\frac{u_i+u_{i+1}}{2},u_{i+1}] \ .
\end{array}\right.$$

Since $E_u$ is obtained from $D_u$ by reparametrizing and adding a path of bounded selfadjoint operators with vanishing endpoints, homotopy invariance implies that
$$\spfl((D_u)_{u\in [0,1]})=\spfl((E_u)_{u \in [0,1]})$$ 
and
$$\ind(\ra_u + D_u)=\ind(\ra_u + E_u) \ .$$
Note that for each $i$ the path $E_u^{U_i}, U_i:=(u_i,\frac{u_{i}+u_{i+1}}{2})$ consists of invertible operators. Hence it is homotopic to the constant path $E_{u_i}$ through paths with invertible endpoints and contributes neither to the spectral flow nor to the index. The previous proposition implies that for the paths $E_u^{V_i}, V_i:=(\frac{u_{i}+u_{i+1}}{2},u_{i+1})$ index and spectral flow agree. 

Now the assertion follows from the additivity of spectral flow and index (see Lemma \ref{cut}) with respect to concatenation of paths.
\end{proof}

In the following we prove a variation of the theorem: instead of a cylinder we consider the unit interval and Atiyah-Patodi-Singer index conditions. The advantage is that we need not assume that the endpoints are invertible. 

We will deal with the von Neumann algebraic tensor product $\cM:=B(L^2(I))\ten\cN$ where $I \subset \R$ is an interval. For $I$ we will have $[0,1]$, $\R$, $[0,\infty)$ or $[-\infty,1)$. The interval will not be reflected in our notation, since it should be clear from the context which interval is meant.

As before, let $(D_u)_{u \in [0,1]}$ be a path of selfadjoint operators with common domain and resolvents in $K(\cN)$ and such that $D_u$ depends continuously on $u$ as a bounded operator from $H(D_0)$ to $H$. We do not assume that $D_0, D_1$ are invertible but we assume that $D_u$ is constant on $[0,\ve)$ and on $(1-\ve,1]$ for some $\ve>0$.

We set $P_u=1_{[0,\infty)}(D_u)$.

We define the unbounded operator $(\ra_u + D_u)^{APS}$ on $L^2([0,1],H)$ as the closure of $\ra_u + D_u$ with domain
$$\{f\in \Cinf([0,1],H(D_0))~|~ P_0f(0)=0, (1-P_1)f(1)=0\}$$
and, similarly, the operator $(-\ra_u + D_u)^{APS}$ as the closure of $-\ra_u + D_u$ with domain
$$\{f\in \Cinf([0,1],H(D_0))~|~ (1-P_0)f(0)=0, P_1f(1)=0\} \ .$$
The operator 
$$\tilde \cD^{APS}=\left(\begin{array}{cc} 0 & (-\ra_u + D_u)^{APS} \\
(\ra_u + D_u)^{APS} & 0 \end{array}\right)$$
is formally selfadjoint.

We also need the case of halfcylinders with Atiyah-Patodi-Singer boundary conditions:

We define the operator $(\ra_u + D_u)^{lAPS}$ as the closure of $\ra_u +D_u$ on $L^2([0,\infty),H)$ with domain
$$\{f \in \Cinf_c([0,\infty),H(D_0))~|~ P_0f(0)=0\} \ .$$ 
Here, as usual, we have set $D_u=D_1$ for $u\ge 1$.

Furthermore we let the operator $(-\ra_u + D_u)^{lAPS}$ be the closure of $-\ra_u +D_u$ with domain
$$\{f \in \Cinf_c([0,\infty),H(D_0))~|~ (1-P_0)f(0)=0\} \ .$$
This is a formal adjoint of $(\ra_u + D_u)^{lAPS}$.

Similarly, we define $(\ra_u + D_u)^{rAPS}$ as the closure of $\ra_u +D_u$ on $L^2((-\infty,1],H)$ with domain
$$\{f \in \Cinf_c((-\infty,1],H(D_0))~|~ (1-P_1)f(1)=0\}$$
and a formal adjoint $(-\ra_u + D_u)^{rAPS}$ as the closure of $-\ra_u +D_u$ on $L^2((-\infty,1],H)$ with domain
$$\{f \in \Cinf_c((-\infty,1],H(D_0))~|~ P_1f(1)=0\} \ .$$
 
\begin{prop}\label{APS Br-Fred}
The operator $\tilde \cD^{APS}$ is selfadjoint with resolvents in $K(\cM)$. In particular it is affiliated to $\cM$ and Breuer-Fredholm.
\end{prop}

\begin{proof}
Without loss of generality we may assume that the endpoints $D_0$, $D_1$ are invertible. This can be seen as follows: first note that it is enough to prove the assertion for a perturbation of $\tilde \cD^{APS}$ by a bounded selfadjoint element of $\cM$. Let $\phi \in \Cinf([0,1])$ be a positive function with $\supp \phi \in [0,\frac 34]$ and $\supp (1-\phi) \in [\frac 14, 1]$. Instead of $D_u$, we may consider the path 
$$\ov D_u:=D_u + \phi(u)(1_{[0,1]}(D_0)-1_{[-1,0)}(D_0)) + (1-\phi(u))(1_{[0,1]}(D_1)-1_{[-1,0)}(D_1)) \ ,$$ 
which has invertible endpoints. It holds that $1_{[0,\infty)}(\ov D_i)=P_i,~i=0,1$. Thus the Atiyah-Patodi-Singer boundary conditions defined using the path $(\ov D_u)_{u \in [0,1]}$ are the same as the ones using $(D_u)_{u \in [0,1]}$.  
 
We adapt the proof of Lemma \ref{afffred}, omitting some details. Let $(U_i)_{i=0,\dots, k+1}$ be a covering of $[0,1]$ by open intervals such that $U_i \subset (\frac{\ve}{2},1-\frac{\ve}{2}), i=1,\dots, k,$ for $\ve$ as above, and $U_0=[0,\ve)$, $U_{k+1}=(1-\ve, 1]$. Furthermore we assume that the path $D^{U_i}_u$ fulfills Assumption \ref{ass} for $i=1,\dots, k$.

We let $(\chi_i^2)_{i=0, \dots, k+1}$ be a partition of unity subordinate to the covering $(U_i)_{i=0, \dots, k+1}$ such that each $\chi_i$ is smooth. 

For $\lambda \in i\R\setminus \{0\}$ and $i\neq 0, k+1$ we define $$Q_i(\lambda):=(\tilde \cD^{U_i} - \lambda)^{-1} \ .$$ 
Parametrices near the endpoints are defined as follows: the operator $(\ra_u + D_0)^{lAPS}$ is invertible with inverse
$$(((\ra_u + D_0)^{lAPS})^{-1}f)(x)$$
$$=\int_0^{\infty} \bigl(-1_{[0,\infty)}(x-y)e^{-(x-y)D_0} P_0 + 1_{[0,\infty)}(y-x)e^{-(x-y)D_0}(1-P_0)\bigr) f(y)~dy \ .$$
(See Prop. 22.4 in \cite{BW} for a detailed discussion of this formula in the case of Dirac operators.) The inverse is in $\cM$.

A similar formula shows that the operator $(-\ra_u + D_0)^{lAPS}$ is invertible as well.
It follows that the operator 
$$\tilde \cD_0^{lAPS}:=\left(\begin{array}{cc} 0 & (- \ra_u + D_0)^{lAPS} \\ (\ra_u + D_0)^{lAPS} & 0 \end{array}\right)$$ is invertible with symmetric inverse in $M_2(\cM)$. In particular it is selfadjoint and affiliated to $M_2(\cM)$.
We set $$Q_0(\lambda)=(\tilde \cD_0^{lAPS} - \lambda)^{-1} \ .$$

Analogously, the operator
$$\tilde \cD_1^{rAPS}:=\left(\begin{array}{cc} 0 & (- \ra_u + D_1)^{rAPS} \\ (\ra_u + D_1)^{rAPS} & 0 \end{array}\right)$$ is selfadjoint, affiliated to $M_2(\cM)$ and invertible.
Set $$Q_{k+1}(\lambda)=(\tilde \cD_1^{rAPS} - \lambda)^{-1} \ .$$ 

Define $$Q(\lambda)=\sum_{i=0}^{k+1}\chi_i Q_i(\lambda)\chi_i \ .$$ 
To see that this operator is in $K(M_2(\cM))$ one uses the following analogue of Lemma \ref{compinv}: since the operators $e^{-xD_0} P_0$ and $e^{xD_0}(1-P_0)$ are in $K(\cN)$ for $x \neq 0$, it follows from the above formula for the inverse that $\phi((\ra_u + D_0)^{lAPS})^{-1}\in K(\cM)$ for $\phi \in \Cinf_c([0,\infty))$, and similarly $\phi((\ra_u + D_1)^{rAPS})^{-1}\in K(\cM)$ for $\phi \in \Cinf_c((-\infty,1])$.

Now one shows as in the proof of Lemma \ref{afffred} that for $|\lambda|$ large enough the operator
$$(\tilde \cD^{APS}-\lambda)Q(\lambda)=:1+K$$
is invertible. 

Thus $Q(\lambda)(1+K)^{-1}$ is a right inverse of $(\tilde \cD^{APS}-\lambda)$. Similarly one constructs a left inverse.

It follows that $(\tilde \cD^{APS}-\lambda)^{-1}=Q(\lambda)(1+K)^{-1} \in K(M_2(\cM))$. This implies the assertion.
\end{proof}

\begin{theorem}\label{spfl=ind APS}
Let $(D_u)_{u \in [0,1]}$ be a path of selfadjoint operators with common domain and with resolvents in $K(\cN)$. We assume that $D_u$ depends continuously on $u$ as a bounded operator from $H(D_0)$ to $H$. Furthermore we assume that the path is constant near each of the endpoints. Then
$$\spfl((D_u)_{u \in [0,1]})=\ind ((\ra_u + D_u)^{APS}) \ .$$
\end{theorem}

A path fulfilling all but the last condition may always be deformed such that it is constant near the endpoints without changing the spectral flow. However, it is not clear from our proof whether for such a more general path the right hand side of the equation is well-defined.

\begin{proof}
By the argument from the beginning of Prop. \ref{APS Br-Fred}, we can assume that $D_0, D_1$ are invertible. (Here we use that the perturbation defined there leaves both sides of the equation unchanged.)
 
In the proof of the previous proposition we saw that $(\ra_u + D_0)^{lAPS}$ and $(\ra_u + D_1)^{rAPS}$ are invertible. In particular they are Breuer-Fredholm with vanishing index. 
It holds that
\begin{align*}
\ind((\ra_u + D_u)^{APS}) &=
\ind((\ra_u + D_u)^{APS})+ \ind(\ra_u+D_0) + \ind(\ra_u+D_1)\\
&=\ind((\ra_u + D_0)^{lAPS}) + \ind(\ra_u + D_u) + \ind((\ra_u + D_1)^{rAPS}) \\
&=\ind(\ra_u + D_u) \\
&=\spfl((D_u)_{u \in [0,1]}) \ .
\end{align*}
Here the second equality follows from a cut-and-paste argument whose proof is as in Lemma \ref{cut}.
The last equality follows from Theorem \ref{spflind}.
\end{proof}

\section{Geometric operators on a foliated manifold}
\label{sign-fol}

In this section we derive some general formulas for the spectral flow of tangential operators for foliations.  The new phenonemon appearing here is that the metric, and thus the von Neumann algebra, may depend on the parameter.

In the following we will use notation and terminology from \cite{MS}.

Let $(M,\mathcal{F})$ be a compact manifold, foliated by an integrable distribution $T\mathcal{F}\subset TM$ of odd dimension $p$. Assume that the foliation is oriented, that is, the bundle $\Lambda^pT^*\mathcal F$ is trivial, and \textit{assume there exists a holonomy invariant transverse measure} $\Lambda$. 

Let $M$ be endowed with a leafwise Riemannian metric (that is, a positive definite element of $C^{\infty}_{tang}(M,S^2T^*\mathcal F)$). The induced leafwise volume form is denoted by $\vol_L$. There is an induced measure $\mu=\vol_L d\Lambda$ on $M$. 

Let $\mathcal R$ be the Borel equivalence relation 
$$
\mathcal R=\{(x,y)\,: x,y\;\;\text{are in the same leaf } L\;\text{of }\;\mathcal F \, \}
$$
with the structure of a measured groupoid given by $\Lambda$.

Let $E$ be a complex vector bundle on $M$ endowed with a hermitian product. We get a field of Hilbert spaces 
$$\mathbf{H}=\{H_x:=L^2(L_x, E_{|L_x})\}_{x\in M} \ .$$ This is endowed with a measurable structure, as explained in the Appendix of \cite{HL2}. The direct integral of the field $\mathbf{H}$ is a separable Hilbert space \cite[p. 172]{Di}. The groupoid
$\mathcal R$ has a natural square integrable representation on $\mathbf{ H}$ given by $$\mathcal R \ni (x,y)\mapsto (\id: H_x\rightarrow H_y) \ .$$ 

We write $\End{}_{\mathcal{R}}(\mathbf{ H})$ for the algebra of uniformly bounded measurable fields of intertwining operators, and define the \emph{von Neumann algebra of the foliation} 
$$W^*(\cF,\cR):=\{[T], T \in \End{}_{\mathcal{R}}(\mathbf{ H}),\text{ where } T_1\sim T_2 \text{ iff equal on $\Lambda$-almost every leaf} \} \ .$$ 
This comes equipped with a semifinite trace $\trl$ \cite[p. 149~ff.]{MS}.

Let $(D_u)_{u\in [0,1]}$ be a path of tangential Dirac operators acting on the sections of $E$ with coefficients depending (for simplicity) smoothly on $u$. Now we also allow the leafwise metric $g_u$ on $M$ and the hermitian product $s_u$ on $E$ to depend smoothly on the parameter $u$. We assume all these paths to be locally constant near $u=0$ and $u=1$.

The closure of $D_u$ -- which we also denote by $D_u$ -- has as domain the measurable field ${\mathbf W}^1$ of Sobolev spaces $W^1(L_x,E_x)$. As topological spaces these are independent of the metric because the leaves are of bounded geometry. Each operator $D_u$ is selfadjoint and affiliated to the von Neumann algebra $W^*(\cF,\cR)$. Since its resolvents are in $K (W^*(\cF,\cR))$, the operator $D_u$ is Breuer-Fredholm. 

Note that the existing definitions \cite{Ph} \cite{Wa} of the spectral flow do not directly apply to the path $(D_u)_{u \in [0,1]}$ since the operators $D_u$ act on different Hilbert fields. We need to trivialize the path of Hilbert fields.

We write ${\mathbf H}^u$ for the above field of Hilbert spaces at the point $u \in [0,1]$. First we identify the hermitian product on the bundle $E$ along the path: we write $E^{s_u}$ for $E$ endowed with the hermitian product $s_u$. There is a unique bundle endomorphism $a_u$ on $E$ such that $s_0(v,a_uw)=s_u(v,w)$ for all $v,w \in E_x,~x \in M$. Since $a_u$ is positive with respect to $s_0$, we can define $T_u=a_u^{1/2}$. Then $T_u:E^{s_u}\to E^{s_0}$ is an isometry. Analogously there is a unique bundle endomorphism $b_u$ on $T\mathcal F$ such that $g_0(v,b_uw)=g_u(v,w)$ for all $v,w \in T_x\mathcal F,~x\in M$. Let $U_u=(\det b_u)^{-1/4}T_u$. This is an endomorphism depending smoothly on $u$ and defining an isometric isomorphism $U_u:{\mathbf H}^u \to {\mathbf H}^0$. 
Note that $U_u \in W^*(\cF,\cR)$. 

We point out the following: as \emph{an algebra} the von Neumann algebra $W^*(\cF,\cR)$ does not depend on the metric, neither does its trace $\trl$. But its involution depends on the metric.

\begin{dfn}
We define the spectral flow $\spfl((D_u)_{u\in [0,1]})$ as the spectral flow of the path $(B_u=U_uD_u U_u^{-1})_{u \in [0,1]}$ of operators acting on ${\mathbf H}^0$. 
\end{dfn}

\begin{prop}\label{spflind geometric}
It holds that
$$\spfl((D_u)_{u\in [0,1]})=\ind((\ra_u +D_u)^{APS})
$$
\end{prop}

The operator $(\ra_u + D_u)^{APS}$ is Breuer-Fredholm with respect to the von Neumann algebra associated to the foliated manifold with boundary $M \times [0,1]$ whose leaves are of the form $L \times [0,1]$ with leafwise metric $du^2 + g_u$. Note that since $g_u$ depends on the parameter, the situation here is different than the one considered in the previous section.

\begin{proof}
From Theorem \ref{spfl=ind APS} we get that
$$\spfl((D_u)_{u\in [0,1]})=\spfl((U_u D_u U_u^{-1})_{u\in [0,1]})=\ind((\partial_u+ U_u D_uU_u^{-1})^{APS}) \ .$$ It holds that 
$$\partial_u+ U_u D_u U_u^{-1}=U_u(\partial_u + U_u^{-1} \partial_u(U_u)+ D_u)U_u^{-1} \ .$$
Since $\supp(U_u^{-1} \partial_u(U_u)) \subset (0,1)$, it follows that
$$\ind((\partial_u+ U_u D_u U_u^{-1})^{APS})=\ind((\partial_u + D_u)^{APS}) \ .$$ 
\end{proof}

\begin{prop}
For $s>0$ it holds that
\begin{align}
\label{integral}
\spfl((D_u)_{u \in [0,1]}) =&\sqrt{\frac{s}{\pi}} \int_{0}^1 \trl \left(\dot{D_u}e^{-sD_u^2}\right)du+\frac{1}{2}\eta_s(D_1)-\\ 
\nonumber &-\frac{1}{2}\eta_s(D_0)+\frac{1}{2}\trl P_{\Ker (D_1)}-\frac{1}{2}\trl P_{\Ker (D_0)} \ .
\end{align} 
\end{prop}
Here $\eta_s(D_i)=\ds\frac{1}{\sqrt \pi}\int_s^\infty \trl (D_i e^{-tD_i^2})\ds\frac{dt}{\sqrt t}$ is the truncated foliated eta invariant, defined in \cite[\S 8]{CP}. The operator $P_{\Ker(D_i)}$ denotes the projection onto the kernel of $D_i$.

\begin{proof}
We want to apply the integral formula Prop. 6.7 in \cite{Wa} to the path $B_u$. Since in general the endpoints of $B_u$ are not invertible, as required in \cite{Wa}, we construct a path with linearly perturbed endpoints. Let $A_0, A_1 \in W^*(\cF,\cR)$ be symmetric such that $B_0+A_0$ and $B_1+A_1$ are invertible, and define
the path 
$$
\a(u)=\left\{\begin{array}{ll}
\a_0(u):=B_0+(u+2)A_0\;,\,& u\in [-2,-1]
\\B_0-uA_0\;,\,& u\in [-1,0] \\B_u\;,\, & u\in [0,1] \\B_1+(u-1)A_1\;,\, & u\in [1,2]\\
\a_1(u):=B_1+(u-2)A_1\;, &\, u\in [2,3] \ ,
\end{array}\right .
$$
whose spectral flow equals the spectral flow of $B_u$.
The piece $\b:=\a_{|[-1,1]}$ has now invertible endpoints.
Then write $$\spfl(B_u)=\spfl (\a)=\spfl(\a_0)+\spfl(\a_1)+\spfl(\b).$$
Applying Prop. 6.7 in \cite{Wa} to the term $\spfl(\b)=\spfl (\sqrt{s} \b)$ we obtain
$$\spfl(B_u)= \spfl(\a_0)+\spfl(\a_1)+\frac{1}{\sqrt{\pi}}\int_{-1}^2\sqrt s \trl \left(\dot{\b}(u)e^{-s\b(u)^2}\right)du +$$
$$+\;\;\;\;\;\frac{1}{2}\eta_1(\sqrt s (B_1+A_1)) -\frac{1}{2}\eta_1(\sqrt s (B_0+A_0))
=$$
$$
= \spfl(\a_0)+\spfl(\a_1)+\frac{1}{\sqrt{\pi}}\int_{-1}^0\sqrt{s}\trl\left(-A_0e^{-s(B_0-uA_0)^2}\right)du+\;\;\;\;$$
$$\;\;\;\;\;+\frac{1}{\sqrt{\pi}} \int_{0}^1\sqrt s \trl (\dot{B}_ue^{-sB_u^2})+ \frac{1}{\sqrt{\pi}}\int_{1}^2\sqrt{s}\trl\left(A_1e^{-s(B_1+(u-1)A_1)^2}\right)du+
$$
$$
\;\;\;\;\;\;  +\frac{1}{2}\eta_s(B_1+A_1)-\frac{1}{2}\eta_s(B_0+A_0).
$$
Applying the integral formula Cor. 8.11 in \cite{CP} to the two linear paths $\a_0$ and $\a_1$ yields
$$
\spfl(\a_0)=\spfl(\sqrt s\a_0)=\frac{1}{\sqrt{\pi}}\int_{-2}^{-1} \sqrt s\trl\left(A_0e^{-s(B_0+(u+2)A_0)^2}\right)du+$$
$$
\;\;\;+
\,\frac{1}{2}\eta_s(B_0+A_0)-\frac{1}{2}\eta_s(B_0)-\frac{1}{2}\tr P_{\Ker (B_0)}
$$
and similarly for $\spfl(\a_1)$; now
combining the formulas we get
\begin{align}
\label{int-f}
\spfl(B_u)&=\frac{1}{\sqrt{\pi}} \int_{0}^1\sqrt s \trl \left(\dot{B}_u e^{-sB_u^2}\right)du+ \\
\nonumber & \quad +\frac{1}{2}\eta_s(B_1)-\frac{1}{2}\eta_s(B_0)-\frac{1}{2}\trl P_{\Ker (B_0)}+\frac{1}{2}\trl P_{\Ker (B_1)} \ .
\end{align}
Now we show that we can substitute $B_u$ with $D_u$. 

The derivatives $\dot{D}_u, \dot U_u$ and $\frac{d}{du}(U_u)$ are determined by their fibrewise action on smooth compactly supported functions on the fibers, so they do not depend on the metric used to define the involutive structure on $W^*(\cF,\cR)$. 

From this, the cyclicity of the trace, and by using 
$$(\frac{d}{du}U_u)U_u^{-1}+  U_u(\frac{d}{du}U_u^{-1})=\frac{d}{du}(U_uU_u^{-1})=0$$ 
one obtains
\begin{align*}
\lefteqn{\trl \left(\dot B_ue^{-sB_u^2}\right)}\\
&=\trl \left(\frac{d}{du}(U_uD_uU_u^{-1})e^{-sU_uD_u^2U_u^{-1}}\right)\\
&= \trl \left((\frac{d}{du}U_u)D_uU_u^{-1}U_ue^{-sD_u^2}U_u^{-1} \right) + \trl \left(U_uD_
u(\frac{d}{du}U_u^{-1})U_u e^{-sD_u^2}U_u^{-1}\right) \\
& \quad + \trl \left(U_u\dot D_uU_u^{-1}U_ue^{-sD_u^2}U_u^{-1}\right) \\
&=\trl \left(\dot D_ue^{-sD_u^2}\right) \ .
\end{align*}
For the remaining terms in (\ref{int-f}), it is clear that we can substitute $B_u$ with $D_u$ so that we get (\ref{integral}).
\end{proof}

\section{The spectral flow of the odd signature operator}

In this section we prove the vanishing of the von Neumann spectral flow for a path of odd tangential signature operators along a path of metrics.

Let us recall the definition of the leafwise \textit{odd signature operator}. We assume that $\dim \cF=p=2l+1$. Let $E:=\Lambda T^*{\mathcal F}\otimes \C$, and let $\t$ be the leafwise chirality grading, $\t\phi:=i^{l+1+k(k+1)}*\phi$, $\phi\in C^\infty(L_x,\Lambda^k T^*\cF_{|L_x})$ (where $*$ is the leafwise Hodge star operator). The \emph{leafwise odd signature operator $D^{sign}$} is defined on $\O^*_{tang}(M)=C_{tang}^\infty(M,E)$ by
\begin{equation*}
D^{sign}=\t d+d \t \ .
\end{equation*}
Now we assume that we have a path of leafwise Riemannian metrics $(g_u)_{u \in [0,1]}$ depending smoothly on the parameter and constant near $u=0,1$. Thus we get a path of chirality operators $\tau_u$, and a path of signature operators $D^{sign}_u$, correspondingly.

\begin{prop}\label{sf=0}
The spectral flow of the path $(D^{sign}_u)_{u \in [0,1]}$ is zero.
\end{prop}
From (\ref{integral})  we have
\begin{equation}
\label{equCG}
\spfl(D^{sign}_u)=\frac{1}{\sqrt{\pi}} \int_{0}^1\sqrt s \trl \left(\dot{D}^{sign}_ue^{-s(D^{sign}_u)^2}\right)du+\frac{1}{2}\eta_s(D^{sign}_1)-\frac{1}{2}\eta_s(D^{sign}_0)
\end{equation} 
since $\displaystyle\trl P_{\Ker (D_1)}=\trl P_{\Ker (D_0)}$ by the homotopy invariance of foliation Betti numbers \cite{HL2}. 
 
Now take the limit for $s\rightarrow \infty$. Clearly $\displaystyle\lim_{s\rightarrow \infty} \eta_s(D_i)=0$.

The following statement, proved in the case of a covering by Cheeger and Gromov in \cite{CG2}, implies the vanishing of the spectral flow.

\begin{lem}[\cite{CG2}] \label{lem-cg} It holds that
\begin{equation}
\label{cg}
\lim_{s\rightarrow \infty} \int_0^1 \sqrt s \trl \left(\dot{D}^{sign}_ue^{-s(D^{sign}_u)^2}\right)du=0 \ .
\end{equation}\end{lem}

\begin{proof}
Let $\delta=\tau d \tau$. Consider the classical decomposition 
\begin{equation*}
\mathbf{ H}=\ker \Delta\oplus \overline{\Im d}\oplus \overline{\Im  \de}
\end{equation*}
and let $P_u$ be the projection onto $\overline{\Im d}$, and $Q_u$ the projection onto $\overline{\Im \de}$. Note that both depend on $u$ via the metric.
%
We now estimate the integrand separately on the subspaces $\overline{\Im d}$ and $\overline{\Im \de}$. Look for example at the estimate of 
$$
\sqrt s \trl\left(\dot{D}_uQ_u e^{-sD_u^2}\right) \ ,
$$ 
where $D_u=D_u^{sign}$. Observe that 
$$
\dot{D}_uQ_u=\dot{\t}_u \,d Q_u=\dot{\t}_u\t_u^{-1}\t_u d Q_u=C_u D_uQ_u
$$
with $C_u \in W(\cF, \cR)$. Then
\begin{align*}
\left|\sqrt s\trl\left(\dot{D}_uQ_u e^{-sD_u^2}\right)\right|&=\left|\sqrt s\trl\left( C_uD_u Q_ue^{-sD_u^2}\right)\right|\\
&\leq\n{C_uQ_u}\left|\sqrt s \trl\left(|D_u| e^{-s D_u^2}\right)\right|
\end{align*}

Let now $E_\la=E_\la(u)$ be the spectral measure associated to the Laplacian $D^2_u$. (To simplify the notation, in the following we will not always write explicitely the dependence on $u$). Now, since $|D_u| e^{-sD_u^2} E_0=0$, we can write for $\mu>0$
$$
|\sqrt s\trl\left(|D_u| e^{-sD_u^2}\right)|\leq |\sqrt s\trl \int_0^\infty \sqrt{\la} e^{-s\la}d(E_\la\cap E_0^\perp)|\leq 
$$
\begin{equation}\label{est}
\leq \sqrt{s}\int_0^\mu \sqrt{\la}e ^{-s\la}d\s+\sqrt{s} \int_\mu^\infty \sqrt{\la} e^{-s\la} d\s= I + II
\end{equation}
where $\s=\trl d(E_\la\cap E_0^\perp)$.  For the term (I), observe that the positive function $f(\la)=\sqrt{\la}e^{-s\la}$ has its maximum in the point $\la=\frac{1}{2s}$, with value $f(\frac{1}{2s})=e^{-\frac{1}{2}}\sqrt{\frac{1}{2s}}$, so that we get
$$\sqrt{s}\int_0^\mu \sqrt{\la}e ^{-s\la}d\s\leq \sqrt{s}e^{-\frac{1}{2}}\frac{1}{\sqrt{2s}}\int_0^\mu d\s=\frac{e^{-\frac{1}{2}}}{\sqrt 2}\trl(E_\mu\cap E_0^\perp)
$$ 
In order to estimate the second term (II) in \eqref{est} we write
$$\int_\mu ^\infty \sqrt\la e^{-s\la}d\s=\int_\mu ^\infty \sqrt\la e^{-(s-1)\la}e^{-\la}d\s
$$ 
The function $g(\la)=\sqrt\la e^{-(s-1)\la}$ has its maximum in the point $\la=\frac{1}{2(s-1)}$. Now we set $\mu=\frac{1}{\sqrt{2(s-1)}}$. For $s$ large we have that  $g(\la)\leq g(\mu)$ on the interval $[\mu, \infty)$, getting
$$
\int_\mu ^\infty \sqrt\la e^{-s\la}d\sigma\leq  \sqrt \mu e^{-(s-1)\mu}\trl(e^{-D_u^2})
$$
Summarizing,
$$
I+II \leq \frac{e^{-\frac{1}{2}}}{\sqrt 2}\trl(E_\mu\cap E_0^\perp)+ \sqrt s \sqrt \mu e^{-(s-1)\mu}\trl(e^{-D_u^2}) \ .
$$
By $s\mu=O(\sqrt s)$ and since $\trl(e^{-D_u^2})$ is uniformly bounded in $u$, the second term vanishes uniformly in $u$ as $s\rightarrow \infty$. In the first term we have that $\lim_{\mu\rightarrow 0}$ $\trl(E_\mu\cap E_0^\perp)\rightarrow 0$. Since $\trl(E_\mu\cap E_0^\perp)$ is uniformly bounded in $u$, we can use Lebesgue theorem and get (\ref{cg}).
\end{proof}

In the classical case of a closed manifold, the vanishing of the spectral flow for a path of signature operators can be deduced from the fact that the projection onto the kernel depends smoothly on the parameter, see for example \cite[\S 8.15]{Me}. For the tangential signature operator on foliations, the smooth dependence of this projection can be proven as in Theorem 2.2 of \cite{GR}. However, in the von Neumann algebraic situation, the vanishing of the spectral flow does not follow from the smooth dependence of the projection onto the kernel, as the following example shows. Consider the path of operators $(D_u:=i\partial_x + u)_{u\in [-1,1]}$ on $H=L^2(\R)$ and let $\cN$ be the von Neumann algebra of $\Z$-equivariant bounded operators on $L^2(\R)$ with its standard semifinite trace $\tr_{\Z}$. For any $u \in [-1,1]$ the kernel of $D_u$ is even trivial. Compute the von Neumann spectral flow via its definition \cite{Ph}: let $1_{\ge 0}(D)=1_{[0,\infty)}(D)$ be the spectral projection onto the positive part of the spectrum; we get 
\begin{multline*}
\spfl((D_u)_{u \in [-1,1]})=\ind{}_{\Z} \left(1_{\ge 0}(D_{-1})1_{\ge 0}(D_{1}):1_{\ge 0}
(D_{1})H\rightarrow 1_{\ge 0} (D_{-1})H\right)=\\
=\Tr_{\Z}\left(1_{[-1,\infty)}(i\partial_x)1_{(-\infty,1]}(i\partial_x)\right)-\Tr_{\Z}\left(1_{[1,\infty)}(i\partial_x)1_{(-\infty,-1)}(i\partial_x)\right)= \\
=\Tr_{\Z} 1_{[-1,1]}(i\partial_x)\neq 0 \ .\;\;\;\;\;\;\;\;
\end{multline*}
By the definition of spectral flow \cite{Ph}, a very direct proof of Prop. \ref{sf=0} could be given if one could show that the projection onto the \emph{positive part} of the spectrum of $D^{sign}_u$ depends continuously on $u$: such a proof is not known.

\section{Measured analytic signature of foliated manifolds with boundary}
As an application of Theorem \ref{spflind geometric} and Prop. \ref{sf=0}, we now prove that the analytic $\Lambda$-signature for a foliated manifolds with boundary does not depend on the metric. 

Let $(M,\cF)$ be a foliated manifold with boundary, with even dimensional leaves and foliation structure transverse to the boundary. Assume that it admits a holonomy invariant transverse measure $\Lambda$. Furthermore we assume that $M$ is endowed with a leafwise Riemannian metric, which is of product type near the boundary.
The tangential signature operator $D^{sign}$ is $\mathbb Z_2$-graded by chirality. We denote by $(D^{sign})^{APS}$ its closure defined by using Atiyah-Patodi-Singer index conditions. The following definition was introduced by Antonini in \cite{An2}. 

\begin{dfn}
The analytic $\Lambda$-signature is by definition the measured $L^2$-index $$\s_{\Lambda, an}(M,\partial M)=\ind_{L^2,\Lambda}(D^{sign,+})=\ind{}((D^{sign,+})^{APS}) +\trl (P_{\Ker D^\partial}) \ ,$$ 
where
$D^\partial$ is the odd signature operator induced on the boundary. 
\end{dfn}

The measured $L^2$-index $\ind_{L^2,\Lambda}(D^{sign,+})$ is defined using the corresponding cylindric setting. We will not use it in the following.

\begin{prop}\label{ind-metric}
$\s_{\Lambda, an}(M,\partial M)$ does not depend on the metric on $M$.
\end{prop}

\begin{proof}
Our proof is inspired by arguments in \cite[\S 6.1]{LP}.

Let $g_u$ be a smooth path of leafwise Riemannian metrics on $M$, which are of product type near the boundary. Let $D^{sign}_u$ be the induced path of foliated signature operators.
The following gluing formula holds: 
\begin{equation}\label{variational APS}
\ind{}((D^{sign,+}_1)^{APS})-\ind{}((D^{sign,+}_2)^{APS})=\ind{}((\partial_u+D^\partial_u)^{APS}) \ .
\end{equation}
This follows from a von Neumann relative index theorem applied to the tangential signature operator on the closed foliated manifold $M_1\cup_{\partial M} ([1,2]\times \partial M)\cup (-M_2)$ where $M_{i}$ is $M$ with metric $g_{i}$, $i\in\{1,2\}$, and the tangential metric on $[1,2]\times \partial M$ is $du^2+g_u$.
The technique to prove the relative index theorem is essentially the one we used  in Lemma \ref{cut}. Alternatively, it follows from Ramachandran's index theorem \cite{Ram}.
We conclude from Theorem \ref{spflind geometric} that
\begin{equation}
\label{aps-sf}
\ind{}((D^{sign,+}_1)^{APS})-\ind{}((D^{sign,+}_2)^{APS})=\spfl((D^\partial_u)_{u\in[1,2]}) \ .
\end{equation}

Now the assertion follows from Prop. \ref{sf=0} by using that $\trl (P_{\Ker D^\partial})$ does not depend on the metric as well, being a homotopy invariant \cite{HL2}.
\end{proof}

Our methods generalize directly to tangential signature operators twisted by a bundle $E$ which is flat near the boundary. However, it seems that the homotopy invariance of the foliated Betti numbers \cite{HL2}, which we used in the proof, has not yet been established for the twisted situation in general (see Theorem 10.6 in \cite{BH} for a partial result). 

Note that then connection on $E$ has to be fixed near the boundary since in general there is spectral flow if the flat connection varies. 

\appendix

\section{On the integral formula for bounded operators}\label{integral formula}

Let $\cN$ be a von Neumann algebra acting on a separable Hilbert space $H$ and endowed with a faithful normal semifinite trace $\tr$. 

We denote by $l^1(\cN)$ the space of operators $K \in \cN$ such that $\tr |K| < \infty$ and endow it with the norm $\|K\|_1=\|K\| + \tr(|K|)$.

In the following a path of bounded operators is not assumed to be continuous in the norm topology. The conditions will be specified. 

We take the opportunity to point out an error in the assumptions of Lemma 6.1 and 6.3, Prop. 6.2 and Theorem 6.4 of \cite{Wa}, which is relevant for the following: There it is always assumed that one deals with a path $(F_u)_{u \in [0,1]}$ in $\cN$ such that $(u \mapsto F_uK) \in C^1([0,1],l^1(\cN))$ for all $K \in l^1(\cN)$. What is needed for the proofs is the stronger assumption that
the map 
$$l^1(\cN) \to C^1([0,1],l^1(\cN)) \ ,$$
$$K \mapsto (u \mapsto F_uK) $$ 
is well-defined and bounded. 
The conclusions of Lemma 6.1 and Prop. 6.2 should be modified accordingly. The arguments still remain valid.

\begin{lem}
Let $(F_u)_{u \in [0,1]}$ be a path in $\cN$.
If the map $$\cJ \to C^1([0,1],\cJ) \ ,$$
$$K \mapsto (u \mapsto F_uK) $$ 
is well-defined and bounded for $\cJ=K(\cN)$, then it is also well-defined and bounded for $\cJ=l^1(\cN)$. 
\end{lem}

\begin{proof} Let $K \in l^1(\cN)$ with polar decomposition $K=|K|U$. 

Let $1_n$ be the characteristic function of $[\frac 1n, \infty)$. Note that $1_n(|K|) \in l^1(\cN)$: Since $1_n(|K|) \in K(\cN)$, the operator $1-1_n(|K|)$ is Breuer-Fredholm. Thus the projection onto its kernel, which is $1_n(|K|)$, has finite trace. It follows that $K_n:=1_n(|K|)K \in l^1(\cN)$. Furthermore $K_n$ converges in $l^1(\cN)$ to $K$.

From this we conclude that the set $S=\{K \in l^1(\cN)~|~1_n(|K|)K=K\}$ is dense in $l^1(\cN)$.

Now let $K \in S$.
   
By assumption, the function 
$$t \mapsto F_t\, 1_n(|K|)$$ 
is bounded in $C^1([0,1],K(\cN))$ by $C \|1_n(|K|)\|=C$ for some $C>0$ independent of $n$ and $K$. It follows that the function 
$$t \mapsto (F_tK=F_t\, 1_n(|K|)K)$$ 
is bounded in $C^1([0,1],l^1(\cN))$ by $C\|K\|_1$. 

This shows the assertion.

\end{proof}  

In \cite{Wa} essentially the following theorem was proven with the additional assumption that each $F_u$ is the bounded transform of an unbounded operator with resolvents in $K(\cN)$. Compare the theorem with the main result in \cite{CPS}.

\begin{theorem}
Let $(F_u)_{u \in [0,1]}$ be a path of selfadjoint Breuer-Fredholm operators in $\cN$ with $\| F_u\| \le 1$ for all $u \in [0,1]$ and such that $F_0,F_1$ are invertible.  

We assume that the map
$$K(\cN) \to C^1([0,1],K(\cN)) \ ,$$
$$K \mapsto (u \mapsto F_uK) $$ 
is well-defined and bounded.

Let $\chi \in C^2_c(\R)$ be an odd function such that $\chi(1)=1$ and $\chi'(0)>0$ and such that $\chi|_{[-1,1]}$ is non-decreasing. 

Assume that $(u\mapsto \chi'(F_u)) \in C([0,1],l^1(\cN))$ and that $(u\mapsto (\chi(F_u)^2-1)) \in C^1([0,1],l^1(\cN))$. 
Then 
\begin{eqnarray*}
\spfl((F_u)_{u \in [0,1]})&=&\frac 12 \int_0^1  \tr \bigl((\frac{d}{du}F_u)\chi'(F_u)\bigr)~ du\\
&& +~ \frac 12 \tr\bigl(2P_1-1- \chi(F_1)\bigr) - \frac 12 \tr \bigl(2P_0-1-\chi(F_0)\bigr)\ ,
\end{eqnarray*} 
where $P_i= 1_{\ge 0}(F_i)$.
\end{theorem}

\begin{proof} 
The two terms in the last line are well-defined by $(2P_i-1)-\chi(F_i)=((2P_i-1)+\chi(F_i))^{-1}(1-\chi(F_i)^2) \in l^1(\cN)$.

Let $(\phi_m)_{m\in \bbbn} \subset C^{\infty}_c(\R)$ be a sequence of even functions with the following properties:
\begin{itemize}
\item $\phi_m$ is non-decreasing on $[-1,0]$ and equals $1$ in a neighbourhood of the origin,
\item $\supp \phi_m$ is contained in the interior of $\supp (\chi^2-1) \cap [-1,1]$,
\item $\phi_m(x) \le \phi_{m+1}(x)\le 1$ for any $x \in \R$ and $\phi_{m+1}|_{\supp \phi_m}=1$,
\item $\phi_m(x)$ converges to $1$ for $m \to \infty$ for all $x$ in the interior of $\supp (\chi^2-1) \cap [-1,1]$.
\end{itemize} 

By Prop. 6.2 in \cite{Wa} the map
\begin{align*}
l^1(\cN)&\to C^1([0,1],l^1(\cN)) \\
K &\mapsto \phi(F_t)K
\end{align*}
is well-defined and continuous.

Since $\phi_m/(\chi^2-1)$ can be extended by zero to an element in $C_c(\R)$ and $\chi(F_u)^2-1 \in l^1(\cN)$, also $\phi_m(F_u) \in l^1(\cN)$.

For $m \in \bbbn$ and $x \in [-1,1]$ let 
$$\chi_m(x)=\frac{1}{C_m}\int_0^x \chi'(y)\phi_m(y) ~dy $$ 
with $C_m=\int_0^1 \chi'(y)\phi_m(y)~dy$. Extend $\chi_m$ to an odd function in $C_c^2(\R)$ such that $\supp(\chi_m^2-1)$ contains a neighbourhood of $\supp (\chi^2-1)$. 

The equality $\chi_m'(F_u)=\frac{1}{C_m}\phi_m(F_u)\chi'(F_u)$ implies by Prop. 6.2 of \cite{Wa} that 
$$(u \mapsto \chi_m'(F_u)) \in C([0,1],l^1(\cN)) \ .$$

Define $g_m \in C_c^2(\R)$ by $g_m(x):=(\chi_m^2(x)-1)/(\chi^2(x)-1)$ if $\chi^2(x) \neq 1$ and $x \in [-1,1]$, and $g_m(x)=0$ if $\chi^2(x)=1$ or $|x| \ge 1$. Since $\chi_m(F_u)^2-1=g_m(F_u)(\chi(F_u)^2-1)$, we infer from Prop. 6.2 in \cite{Wa} that $$(u \mapsto (\chi_m(F_u)^2-1)) \in C^1([0,1], l^1(\cN))$$ and then from Lemma 6.3 in \cite{Wa} that 
$$(u \mapsto (e^{\pi i(\chi_m(F_u)+1)}-1)) \in C^1([0,1], l^1(\cN))\ .$$ 

Now we show that the integral formula holds for $\chi_m$. In the end we will take the limit over $m$. Since the proof is very similar to the proof of Theorem 6.4 in \cite{Wa}, we will omit some details.

Let $G_u:=(1-u)(2P_0-1) + u\chi_m(F_0)$ and $H_u:=(1-u)\chi_m(F_1)+ u(2P_1-1)$ and 
define $(Q_u)_{u \in [-1,2]}$ by $Q_u:= \chi_m(F_u)$ for $u \in [0,1]$, $Q_u:=G_{u+1}$ for $u \in [-1,0]$ and $Q_u:=H_{u-1}$ for $u \in [1,2]$. 

It holds that
$$\frac{1}{2 \pi i} \int_0^1 \tr  \bigl(e^{-\pi i(G_u+1)}\frac{d}{du}e^{\pi i (G_u+1)}\bigr)~du + \frac{1}{2 \pi i} \int_0^1 \tr  \bigl(e^{-\pi i(H_u+1)}\frac{d}{du}e^{\pi i (H_u+1)}\bigr)~dt$$
$$=\frac 12 \tr\bigl((2P_1-1)- \chi_m(F_1)\bigr)- \frac 12 \tr\bigl((2P_0-1)- \chi_m(F_0)\bigr) \ .$$

Hence
\begin{align*}
\spfl((F_u)_{u \in [0,1]})&=
\spfl((Q_u)_{u \in [-1,2]}) \\ 
&=\frac{1}{2 \pi i} \int_0^1 \tr  \bigl(e^{-\pi i(\chi_m(F_u)+1)}\frac{d}{du}e^{\pi i ( \chi_m(F_u)+1)}\bigr)~du\\
&\quad + \frac 12 \tr\bigl((2P_1-1)- \chi_m(F_1)\bigr)- \frac 12 \tr\bigl((2P_0-1)- \chi_m(F_0)\bigr) \ .
\end{align*}

We evaluate the integral.

Set $f:=e^{\pi i (\chi_m+1)}-1$. Since 
$$\bigl(\supp f\cap [-1,1]\bigr) \subset \bigl(\supp (\chi_m^2-1) \cap [-1,1]\bigr) \subset \supp \phi_m \ ,$$ 
it holds that $\phi_{n}|_{\supp f\cap [-1,1]}=1$ for $n>m$. This implies that
\begin{align*}
\frac{d}{du} f(F_u)&= \frac{d}{du}\bigl(\phi_{n}(F_u)f(F_u)\bigr)\\
&=\bigl(\frac{d}{du}\phi_{n}(F_u)\bigr)f(F_u)+\phi_{n}(F_u)\frac{d}{du}f(F_u)  \ . 
\end{align*}

Thus
$$\tr \bigl(e^{- \pi i (\chi_m(F_u)+1)} \frac{d}{du} f(F_u)\bigr) =\tr \bigl(e^{- \pi i (\chi_m(F_u)+1)}(\frac{d}{du}\phi_{n}(F_u)f(F_u)+\phi_{n}(F_u)\frac{d}{du}f(F_u))\bigr) \ .$$

Let $n>m+2$.

Now
\begin{align*}
\lefteqn{\tr \bigl(e^{- \pi i (\chi_m(F_u)+1)}(\frac{d}{du}\phi_{n}(F_u))f(F_u)\bigr)}\\
 &=
\tr \bigl(e^{- \pi i (\chi_m(F_u)+1)}\phi_{n-1}(F_u)(\frac{d}{du}\phi_{n}(F_u))f(F_u)\bigr)\\
&=\tr \bigl(e^{- \pi i (\chi_m(F_u)+1)}\frac{d}{du}(\phi_{n-1}(F_u)\phi_{n}(F_u))f(F_u)\bigr)\\
& \quad - \tr \bigl(e^{- \pi i (\chi_m(F_u)+1)}(\frac{d}{du}\phi_{n-1}(F_u))\phi_{n}(F_u)f(F_u)\bigr)\ .
\end{align*}
These two terms cancel out since they both equal 
$\tr \bigl(e^{- \pi i (\chi_m(F_u)+1)}(\frac{d}{du}\phi_{n-1}(F_u))f(F_u)\bigr)$:

The first by $\phi_{n-1}\phi_n=\phi_{n-1}$, and the second by $\phi_nf|_{[-1,1]}=f|_{[-1,1]}$. 

It follows that 
$$\tr \bigl(e^{- \pi i (\chi_m(F_u)+1)} \frac{d}{du} f(F_u)\bigr) =\tr \bigl(e^{- \pi i (\chi(F_u)+1)}(\frac{d}{du}f(F_u))\phi_n(F_u)\bigr) \ .$$

A calculation as in the proof of Theorem 6.4 in \cite{Wa} yields that
$$\tr \bigl(e^{- \pi i (\chi_m(F_u)+1)} (\frac{d}{du} f(F_u)) \phi_n(F_u)\bigr) =i\pi ~\tr((\frac{d}{du}F_u)\chi_m'(F_u)\phi_n(F_u)\bigr) \ .$$

Using that $\chi_m'(F_u)\phi_n(F_u)=\chi_m'(F_u)$ we obtain
\begin{align*}
\spfl((F_u)_{u \in [0,1]})& = \frac{1}{2}  \int_0^1  \tr \bigl((\frac{d}{du}F_u)\chi_m'(F_u)\bigr) ~du \\
& \quad + \frac 12 \tr\bigl((2P_1-1)- \chi_m(F_1)\bigr)- \frac 12 \tr\bigl((2P_0-1)- \chi_m(F_0)\bigr) \ .
\end{align*}

Now we consider the limit $m \to \infty$.

First note that $C_m$ converges to $1$. We fix $u$. The sequence of positive operators $C_m\chi_m'(F_u)=\chi'(F_u)\phi_m(F_u) \in l^1(\cN)$ is non-decreasing with supremum $\chi'(F_u)$. Thus $C_m\tr\bigl((\frac{d}{du}F_u)\chi_m'(F_u)\bigr)$ converges to $\tr\bigl((\frac{d}{du}F_u)\chi'(F_u)\bigr)$.

Furthermore for some $C>0$, independent of $m$ and $u$,
\begin{align*}
C_m|\tr\bigl((\frac{d}{du}F_u)\chi_m'(F_u)\bigr)| &\le  \|(\frac{d}{du}F_u) \phi_m(F_u)\| \,\tr(\chi'(F_u)) \\
&\le C \tr(\chi'(F_u)) \ .
\end{align*}
 
Thus by Lebesgue's Lemma $\int_0^1  \tr \bigl((\frac{d}{du}F_u)\chi_m'(F_u)\bigr) \, du$ converges to   $\int_0^1  \tr \bigl((\frac{d}{du}F_u)\chi'(F_u)\bigr) \, du$.

Furthermore $C_m(1+\chi_m(x)) \le 1+\chi(x)$ for $x<0$ and $C_m(1-\chi_m(x)) \le 1-\chi(x)$ for $x>0$. It follows that
$$C_m(1- (2P_i-1)\chi_m(F_i)) \le 1 - (2P_i-1)\chi(F_i) \ .$$
Both sides are positive. The left hand side converges to the right hand side for $m \to \infty$ in $\cN$, hence also in $l^1(\cN)$. Since $$(2P_i-1)-\chi_m(F_i)=(2P_i-1)(1- (2P_i-1)\chi_m(F_i)) \ ,$$ 
it follows that $\tr\bigl((2P_i-1)-\chi_m(F_i)\bigr)$ converges to $\tr\bigl((2P_i-1)-\chi(F_i)\bigr)$.
\end{proof}

\end{document}